\documentclass[11pt]{amsart}
\theoremstyle{plain}
\usepackage{amsmath, amsfonts, amsthm, amscd, mathrsfs}
\usepackage{graphics}
\usepackage{enumerate}
\usepackage{color}
\usepackage{epsfig}
\usepackage[all]{xy}
\usepackage{graphicx,amssymb}
\usepackage{tikz}

\graphicspath{ {Figures/} }
\numberwithin{figure}{section}
\usepackage{fullpage}

\newcommand{\Addresses}{{
  \bigskip
  \footnotesize
   James Farre, \textsc{Department of Mathematics, Yale University} \par\nopagebreak
  \textit{E-mail address}: \texttt{james.farre@yale.edu}
  
    \textsc{Mathematisches Institut, Ruprecht-Karls Universit\"at Heidelberg}\par\nopagebreak
    \textit{E-mail address}:
  \texttt{jfarre@mathi.uni-heidelberg.de}
  
    \noindent Franco Vargas Pallete, \textsc{Department of Mathematics, Yale University}\par\nopagebreak
  \textit{E-mail address}: \texttt{franco.vargaspallete@yale.edu}
  }}

\newcommand{\HH}{\mathbb{H}}

\newcommand{\CC}{\mathbb{C}}
\newcommand{\ZZ}{\mathbb{Z}}
\newcommand{\RR}{\mathbb{R}}

\DeclareMathOperator{\Area}{Area}

\newtheorem{theorem}{Theorem}[section]
\newtheorem{lem}[theorem]{Lemma}
\newtheorem{remark}[theorem]{Remark}
\newtheorem{prop}[theorem]{Proposition}

\newtheorem{cor}[theorem]{Corollary}

\begin{document}
\title{Minimal area surfaces and fibered hyperbolic $3$-manifolds}
\author{James Farre}
\author{Franco Vargas Pallete}
\thanks{J. Farre's research was supported by NSF grant DMS-1902896. F. Vargas Pallete's research was supported by NSF grant DMS-2001997. This work was also supported by the National Science Foundation under Grant No. DMS-1928930, while Farre and Vargas Pallete participated in a program hosted by the Mathematical Sciences Research Institute in Berkeley, California, during the Fall 2020 semester.}
\maketitle
\begin{abstract}
By work of Uhlenbeck, the largest principal curvature of any least area fiber of a hyperbolic $3$-manifold fibering over the circle is bounded below by one. 
We give a short argument to show that, along certain families of fibered hyperbolic $3$-manifolds, there is a uniform lower bound for the maximum principal curvatures of a least area minimal surface which is greater than one.
\end{abstract}

\section{Introduction}

Since Thurston's work on surfaces, the study of hyperbolic manifolds in dimensions two and three has seen an explosion of progress.  Abundant classes of surfaces, such as  pleated surfaces, can be used to probe the geometry of a hyperbolic $3$-manifold at nearly any point.  In this note, we are interested in geometric aspects of \emph{minimal} surfaces in hyperbolic $3$-manifolds, which are far less ubiquitous than pleated surfaces and which  subtly influence the geometry of the space that they inhabit.  

In her seminal work, Uhlenbeck \cite{Uhlenbeck83} investigated minimal immersions of closed surfaces into complete hyperbolic $3$-manifolds with principal curvatures bounded in absolute value by one.
The covering space associated to such a surface is remarkably well behaved in many aspects; the inclusion of the minimal immersion is an incompressible least area minimal embedding, and no other closed minimal surface of any kind can be found in this manifold. 
It has been unclear, geometrically, how far the class of analytically defined \emph{almost-Fuchsian} manifolds, i.e. those which deformation retract onto a minimal surface with principal curvatures bounded by one, could be from Fuchsian.  Very recently, Huang and Lowe \cite{HuangLowe21} proved that the closure of the almost-Fuchsian locus is contained in the well-studied \emph{quasi-Fuchsian} space.  

Following ideas of Hass \cite{Hass15}, we observe that a surface minimizing area in its homotopy class cannot penetrate too deeply into regions of small injectivity radius in a hyperbolic $3$-manifold. In particular, many sequences of incompressible least area minimal surfaces of a fixed topological type are \emph{uniformly thick}, and so minimal limits with the same topology are easily extracted.  

Using this geometric control, we bound the maximum principal curvatures of certain families of minimal surfaces in hyperbolic mapping tori strictly away from one. As the Virtual Fibering Theorem of Agol and Wise asserts that every closed hyperbolic $3$-manifold has a finite cover that fibers over the circle, our applications address the behavior of an important class of least area minimal immersions in closed hyperbolic $3$-manifolds.
%This question motivated the following type of problem: if a family of minimal surfaces are "reasonably far" from being Fuchsian (i.e. totally geodesic) then there should be a uniform lower bound for the maximum principal curvature. 

\begin{theorem}\label{thm:intro1}
Given a hyperbolic $3$-manifold fibering over the circle with monodromy $\psi:\Sigma\rightarrow\Sigma$ and a simple curve $\alpha \subset \Sigma$ such that $i(\alpha,\psi(\alpha)) > 5|\chi(\Sigma)|$, any sequence of hyperbolic mapping tori that drills $\alpha$ admits a uniform lower bound $\mu>1$ for the largest principal curvature in the area minimizer in the homotopy class of the fiber.
\end{theorem}
The examples of Theorem \ref{thm:intro1} are very common.  One need only pass to a power of $\psi$ to ensure that the condition on the intersection of a curve and its image under the monodromy.
We can also bound the principal curvature of area minimizers homotopic to the fiber if a pair of intersecting curves become short along the sequence.
\begin{theorem}\label{thm:intro2}
Given a hyperbolic mapping torus and two intersecting simple curves $\alpha$ and $\beta$ in the fiber, for any sequence of mapping tori that drills $\alpha$ and makes $\beta$ sufficiently short, there is a uniform lower bound $\mu>1$ for the largest principal curvature in the area minimizer in the homotopy class of the fiber.
\end{theorem}

See Section \ref{sec:mapping torus} for precise statements.  
We note that  Huang and Lowe \cite{HuangLowe21} prove a more general version of our applications (in fact, they solve the problem  that originally motivated this investigation). Nevertheless, our results are ``hands on'' in nature and provide information about how curves of short (complex) length in a hyperbolic $3$-manifold affect least area minimal immersions.  Our main technical result provides a short argument for an improved version of a main result of \cite{HuangWang}.  This improvement allows us to easily construct the examples given in Theorems \ref{thm:intro1} and \ref{thm:intro2}. % is  we provide is similar in spirit to the main res    %and are essentially orthogonal to their (powerful) analytical arguments.  %that we prove our applications by different methods than theirs, without proving their main result.

The article is organized as follows. In Section \ref{sec:minimizers} we show that area minimizers cannot go arbitrarily deep into thin parts of controlled shape. In Section \ref{sec:AF} we discuss Uhlenbeck's work on almost-Fuchsian manifolds.  We show, by refocusing perpendiculars at infinity, that a least area homotopy equivalence with principal curvatures between $-1$ and $1$ in a manifold with parabolic cusps contains horocyclic segments, which are closed geodesics in the induced metric. Section \ref{sec:mapping torus} contains our main applications regarding area minimizers in sequences of mapping tori. In Appendix \ref{appendix} we present a known result explaining how to rescue a local area-minimizer from a geometric limit back to a sequence of approximations. 

\section*{Acknowledgments} The second author would like to thank Marco A. M. Guaraco, Joel Hass, and Vanderson Lima for helpful conversations about this work.  Both authors would like to thank Zeno Huang for insightful comments and questions related to this work.

\section{Area minimizing minimal surfaces and short curves}\label{sec:minimizers}

Fix a constant $\epsilon_3$ smaller than the $3$-dimensional Margulis constant.  If $\gamma\subset M$ is a closed geodesic with (real) length smaller than $\epsilon_3$, then $\gamma$ is the core geodesic of a \emph{Margulis tube} $\mathbb T(\epsilon_3)$, the set of all points near $\gamma$ through which there is a non-contractible loop of length at most $2\epsilon_3$, foliated by flat tori $T_r$ at distance $r\le r_0$ from $\gamma$.
We let $\lambda$ denote the \emph{complex length} of $\gamma$ whose real part is length and whose imaginary part lies in $[0,2\pi)$, measuring the twisting angle.
By \cite{BrooksMatelski82}, there is a constant $C = C(\epsilon_3)$  that  satisfies
\begin{equation}\label{eqn:tube_radius}
e^{r_0}|\lambda|\geq C^{-1};
\end{equation}
we call $r_0$ the \emph{radius} of $\mathbb T(\epsilon_3)$. 

The following technical result is one of the main tools of the present article, namely that $\pi_1$-injective area minimizers do not penetrate too deeply into Margulis tubes about short curves. 
Here ``deep" depends on the shape of the boundary of the tube, so that our uniform bound on distance from the area minimizer to the boundary becomes interesting as the radius tends to infinity. This was done originally by Hass in \cite{Hass15} for rank-2 cusps; we follow the same ideas for tubes with a short geodesic core.
\begin{prop}\label{prop:nodeepness}
%Let $\epsilon_3>\epsilon>0$ be a small enough constant. Then there exist constants $\delta>0,\,K>1$ so that for any given $\pi_1$-injective area minimizer surface $\Sigma$ in a compact hyperbolic $3$-manifold $M$ and any given $\epsilon$-Margulis tube $T$ in $M$ whose loxodromic core has complex length $\lambda$ satisfying $Re(\lambda)<\delta$, $K^{-1}\leq Im(\lambda)^2/Re(\lambda) \leq K$, we have that $\Sigma\cap T = \emptyset$.\\
Given $K>1$, there exist constants $\delta(K)>0$ and  $d(K)>0$ such that for any $\pi_1$-injective area minimizing surface $\Sigma$ in a complete hyperbolic $3$-manifold $M$ and $\gamma$ a geodesic in $M$ with complex length $\lambda$ satisfying $Re(\lambda)<\delta$ and  $K^{-1}\leq Im(\lambda)^2/Re(\lambda)$, then we have that $\Sigma\cap M^{<\epsilon_3}$ is within distance $2d$ of $\partial M^{<\epsilon_3}$.
\end{prop}

\begin{proof}

The foliation of a compact component of the $\epsilon_3$-thin part $M^{<\epsilon_3}$ by tori equidistant from the core curve is mean-convex, so we can assume without loss of generality that $\Sigma\cap\partial M^{<\epsilon_3}\neq\emptyset$.

By \cite[Section 3.2]{Minsky10} the boundary $T_r$ of the $r$-neighbourhood around $\gamma$ is a flat torus isometric to $\CC/t_r(\ZZ+\omega_r\ZZ)$, where the parameters $t_r>0, \omega_r\in\CC$ are given by the equations
\[    t_r|\omega_r| = 2\pi\sinh(r)\]
and
\[
    it_r|\omega_r|/\omega_r = Re(\lambda)\cosh(r) + iIm(\lambda)\sinh(r).\]
From this is not hard to see that the area $A_r$ of $T_r$ is given by
\[    A_r = t_r^2|Im(\omega_r)| = 2\pi Re(\lambda)\cosh(r)\sinh(r),
\]
while the injectivity radius $inj_r$ of $T_r$ is bounded by
\[    inj_r \geq \min\lbrace t_r, t_r.Im(\omega_r) \rbrace = \min\lbrace t_r, \frac{2\pi Re(\lambda)\cosh(r)\sinh(r)}{t_r} \rbrace.
\]

Since $|\lambda|\sinh(r) \leq t_r \leq |\lambda|\cosh(r)$ we can further simplify this inequality to
\[    inj_r \geq \min \lbrace |\lambda|\sinh(r), \frac{2\pi Re(\lambda) \sinh(r)}{|\lambda|} \rbrace = \frac{\sinh(r)}{|\lambda|}.\min \lbrace |\lambda|^2, 2\pi Re(\lambda) \rbrace \geq \frac{Re(\lambda)\sinh(r)}{K|\lambda|},
\]
where in the last inequality we have used that $Im(\lambda)^2 \geq K^{-1}.Re(\lambda)$.

As for upper bounds, we have

\begin{equation}\label{eq:upperboundinj}
    inj_r \leq \sqrt{\frac{A_r}{\pi}} = \sqrt{2Re(\lambda)\cosh(r)\sinh(r)}.
\end{equation}

%The $\epsilon$-Margulis tube $\mathbb T(\epsilon)$ around a curve of length $\ell<\delta_0$ corresponds to the $r$-neighbourhood around that curve for some $r$. 
%We let $r_0$ be such that the tube $\mathbb T(\epsilon_3)$ is the $r_0$ neighborhood of the core geodesic.
%Let then $r_0$ be the radius corresponding to $\epsilon_3$

Observe then that for $r_0 > b\ge1.42$, we have
\begin{equation}\label{eq:coareabound}
    \begin{split}
        \int_{b-1}^b 2.inj_r dr &\geq \int_{b-1}^b 2\frac{Re(\lambda)\sinh(r)}{K|\lambda|}  dr\\
        & = \frac{2Re(\lambda)}{K|\lambda|}(\cosh(b)-\cosh(b-1))\\
        & \geq \frac{Re(\lambda)\cosh(b))}{K|\lambda|},
    \end{split}
\end{equation}
where $b\ge1.42$ ensures that the last inequality holds.
With this we can prove that the disk components of $\Sigma\cap \mathbb T(\epsilon_3)$ cannot go deep into the thin part.

\begin{lem}\label{lemma:disk}
Suppose $\Sigma \cap \mathbb T(\epsilon_3)$ is a union of disks. Then as long as $Re(\lambda)$ is small enough, there is $d=d(K)>0$ such that $\Sigma$ stays within distance $d$ from $\partial \mathbb T(\epsilon_3)$.
\end{lem}
\begin{proof}
Denote by $D$ one of the components of $\Sigma \cap \mathbb T(\epsilon_3)$ and let $\ell$ be the length of $\partial D$. Lift the disk $D$ to $\HH^3$. 

First, we observe that as long as $Re(\lambda)$ is small enough, then $\partial D$ cannot be isotopic to the meridian curve of $\partial \mathbb T(\epsilon_3)$.  Indeed, $D$ would be a minimal disk passing through the core geodesic of a tube with very large radius.  A monotonicity argument of Anderson \cite{Anderson82} then guarantees that the area of the disk is at least that of a geodesic disk in the hyperbolic plane of the same radius.  However, since the Gauss curvature of $\Sigma$ is bounded above by $-1$, the area of $\Sigma$ is bounded above by a constant depending only on the topology of $\Sigma$, so the area of this disk would be too large compared to the total area of $\Sigma$.  

The boundary of an $r$-neighborhood of the universal cover of a tube (with core passing from $0$ to $\infty$, say) in the upper half space model for $\HH^3$ is a cone making very small angle with with the complex plane at $0$.  Take a point $x_0\in\partial D$ and a horoball outer-tangent to the boundary of the universal cover of $\mathbb T(\epsilon_3)$ at $x$.  The hyperbolic diameter of $\partial D$ is at most $\ell/2$.  After enlarging our horoball to include the horospheres centered at the same point distance at most $\ell$ away, we find a horoball $B$ such that $\partial D \subset B$.  
Indeed, the boundary of our cone is nearly parallel to the complex plane, so this follows from continuity by looking at the intersection pattern between a horoball tangent to the plane and a horosphere centered at in point at infinity in the half space model for $\HH^3$.  
If $D$ were not contained in $B$, then we could enlarge $B$ further to a horoball $B'$ which contains $D$, and whose closure is tangent to $D$ at a point.  However, the mean curvature of $D$ is $0$ and the mean curvature of the boundary of $B'$ is $1$, which contradicts the maximum principle; thus  $D\subset B$.

Recall that $\mathbb T(\epsilon_3)$ is the $r_0$-neighborhood of its core curve $\gamma$, so that $T_{r_0} = \partial \mathbb T(\epsilon_3)$.  We claim that $r_0\le \log\frac{1}{|\lambda|} + \log C_1$ for some $C_1>0$. Indeed, we can take a point $x$ at distance $r_0$ from the core geodesic so that the geodesic segment joining $x$ and $\gamma.x$ has length $2\epsilon_3$. Denote then $y$ the point at distance $r_0$ with the same projection as $\gamma.x$ that is in the same plane as $x$ and the axis. Then $d(x,y)\approx e^{r_0}Re(\lambda)$, and  $d(y,\gamma.x)\approx e^{r_0}Im(\lambda)$. 
There is an upper bound \cite[Corollary A.2]{HuangWang}, depending only on the topology of $\Sigma$, for the ratio \[\frac{Im(\lambda)^2}{Re(\lambda)}\le 2\pi|\chi(\Sigma)| \sqrt{\frac{4\pi}{\sqrt{3}}}.\]
Using this bound, our assumption $K^{-1}\leq Im(\lambda)^2/Re(\lambda)$, and the triangle inequality can find $C_1$ with $e^{r_0}|\lambda| \leq C_1\epsilon_3$.  The claim follows after taking logarithms. 

Together with (\ref{eq:upperboundinj}), this implies that $inj_{r_0}$ is bounded by a constant depending on $K$. If $\ell<2inj_{r_0}$, then the maximum principle argument shows that  $D$ stays at distance $\ell<2inj_{r_0}$ from $\partial \mathbb T(\epsilon)$, from which the lemma follows. 

%Note that $inj_{r_0} = inj (T_{r_0}) \le \epsilon$, by definition of the Margulis tube.
%Furthermore, denote by $i_\epsilon$ the injectivity radius of $T^\epsilon$, which we assume without loss of generality that is less than a constant $K_1$ dependent on $K$.
 %Moreover, if $\ell<2inj_{r}$ 

%Take $r<r_0$  and $r_0-r<K_2$, a constant dependent on $K$ to be determined. Then if for any $r$ we have that $\ell(\Sigma\cap\partial T_r)<2inj_r$, by the same arguments we have that the disk $D$ stays at distance $K_2+2inj_r+1<2(K_1+K_2)+1$ of $\partial \mathbb T(\epsilon)$, which finishes the argument. 

Suppose $1<b< r_0$ is such that $\ell(\Sigma\cap T_r)\geq 2inj_r$ for all $r\in [b-1, b]$.   By \eqref{eq:coareabound} and the coarea formula we have that

\[    \Area(D)\geq \int_{b-1}^b 2.inj_r dr \geq\frac{Re(\lambda)\cosh(b)}{K|\lambda|}.
\]

Since $D$ is area minimizer with respect to its boundary, it has less area than the disk bounded by $\partial D$ in $T_b$, which in turn is less than the total area of $T_b$. Hence 

\[A_b = 2\pi Re(\lambda)\cosh(b)\sinh(b) \geq \frac{Re(\lambda)\cosh(b)}{K|\lambda|}.
\]

From the above inequality, we obtain $\sinh(b)\geq \frac{1}{2\pi K|\lambda|}$, which in turn implies that $b\ge \log\frac{1}{|\lambda|} - \log \pi K$. Since we also had that $r_0\le \log\frac{1}{|\lambda|} + \log C_2$ for some $C_2(K)>0$, then it follows that $r_0-b\le \log C_2\pi K$.

%On the other hand, from \eqref{eqn:tube_radius} \textcolor{red}{Fix the bound as per our discussion} we see that $r_0\le \log\frac{1}{|\lambda|} +\log(C)$ and hence that $r_0-b\le \log C\pi K$.  

Finally, if there is an $r_0 > r \ge r_0 -\log C_2\pi K$ such that $\ell(\Sigma\cap T_r)<2inj_r< 2\epsilon$, then our previous argument shows that $D$ stays at distance at most $ 2\epsilon + \log C_2\pi K$ from $T_{r_0}$.  Combining all of our inequalities, we may then take $d =  2\epsilon + \log C_2\pi K+1$, which satisfies the conclusion of the lemma.
\end{proof}

\begin{remark}
In general $\Sigma\cap T_r$ may not be connected or differentiable for all $r$, and the topology of a component of the intersection of $\Sigma$ with a solid tube maybe not be a disk. %This was important for the coarea inequality and for the maximum principle argument.
However, by Sard's Theorem, for almost every $r$ the intersection $\Sigma\cap T_r$ is differentiable, while for critical values it has $0$ area, and hence we can apply the coarea inequality by bounding the total length of the intersection for regular values of $r$. 
%\textcolor{red}{All intersections of $D$ with the `horotori' are (connected) circles (as long as they are smooth); use the maximum principal.  Length bounds diameter.}
%\textcolor{blue}{Now I think it could be many disjoint disks, but it works just fine}
%For the maximum principle argument to work we need to bound the diameter of $\Sigma\cap \partial T_r$, and technically we have only bounded its total length. 
%For this to be a problem we will need to have that $\Sigma\cap \partial T_r$ is disconnected. Because $\Sigma$ is incompressible, each component of $\Sigma\cap \partial T_r$ bounds a disk in $\Sigma$, and by maximum principle such minimal disk belong to $T_r$. Then $\Sigma\cap T_r$ is a union of disks and we run the argument on each of them.
To apply the maximum principle argument, we need to bound the diameter of $\Sigma\cap  T_r$.  Since $\Sigma$ is incompressible, each component of $\Sigma\cap T_r$ bounds a disk in $\Sigma$, and by maximum principle such a minimal disk belongs to the solid tube of radius $r$ about $\gamma$. 
Thus the length of each component of  $\Sigma\cap T_r$ bounds the diameter of that component, and we can run the argument on each of them.
%a component of the intersection of $\Sigma$ with a tube that intersects $\Sigma\cap T_r$ in multiple components that have arbitrarily large diameter as a set. This would be a problem if we wanted to apply the maximum principle argument to $\Sigma\cap C^r$. Instead, take $\Sigma\cup T^r$  in $D$. Since $D$ is a disk, any closed curve $\gamma$ in $\Sigma\cap T^r$ bounds a disk $D_\gamma$ in $D$. Then the maximum principle argument applies to $D_\gamma$, since now $\ell(\gamma)$ is a bound for the diameter of $D_\gamma$. Finally, all $D_\gamma$ cover $\Sigma\cap C^r$  (in fact outermost curves are enough). 
%[More rough argument] Large values of $\ell$ add to more area of $D$, which collides withe the are minimizer aspect. Small values of $\ell$ guarantee both that components are eventually disks and that they stay out of the tube at a fixed distance. Large and small is versus the injectivity radius of the boundary.
\end{remark}
%We want this last term to be greater than $A_b = 2\pi Re(\lambda)\cosh(b)\sinh(b)$, which is true if and only if $\sinh(b) \leq \frac{1}{2\pi.K|\lambda|}$. Since $|\lambda| \leq \sqrt{Re(\lambda)^2 + K.Re(\lambda)} \leq \sqrt{2K.Re(\lambda)}$. Take then $b = \sinh^{-1}(\frac{1}{2\pi.K\sqrt{2K.Re(\lambda)}})$, so we have constant $C>1$ so that $C^{-1}<e^{2b}.Re(\lambda)<C$.

We now continue with the rest of the proof of Proposition \ref{prop:nodeepness}.
By Lemma \ref{lemma:disk}, we now only need to show that the non-disk components of $\Sigma$ in $\mathbb T(\epsilon_3)$ are at bounded distance (depending on $K$) from $\partial M^{<\epsilon_3}$. If $r$ is at bounded distance (depending on $K$) from $r_0$, it is enough to show that the non-disk components of the intersection of $\Sigma$ with the $r$-neighborhood of a short geodesic are contained in a bounded neighborhood of $T_r$. 

Take then a component of $\Sigma\cap \mathbb T(\epsilon_3)$ that is not a disk. Take $d$ from Lemma \ref{lemma:disk} and $r\in (r_0-d,r_0)$. Then we can assume that the total length of $\Sigma\cap  T_r$ is bounded below by $2.inj_r$, since otherwise each component of the intersection of $\Sigma$ with the $r$-neighborhood of the short curve bounds a disk and we are finished. Moreover, we can further assume that every component of $\Sigma\cap T_r$ is homotopically essential in $T_r$. Indeed, if such a component was null homotopic then it would bound a disk in $\Sigma$, because $\Sigma$ is incompressible. But then by maximum principle such minimal disk cannot exit $\mathbb T(\epsilon_3)$, so we can take care of it by Lemma \ref{lemma:disk}. Now we have that the length of each component of $\Sigma\cap  T_r$ is bounded below by $2.inj_r$.

By \eqref{eq:coareabound} we see again that $\Area(\Sigma\cap (\bigcup_{s\in (r-1, r)}T_s))$ is bigger than $\Area(T_r)$ for some $r\in (r_0-d,r_0)$, as we did in Lemma \ref{lemma:disk}. Then since $\Sigma$ is incompressible and the curves in $\Sigma\cap T_b$ are mutually disjoint and homotopically essential, each component of $\Sigma\cap T_b$ bounds an essential annulus. But this is impossible since after such homotopy we would have reduced the area of $\Sigma$.
%\textcolor{red}{Topology of intersections again: annuli all have two incompressible curves on the boundary torus.  Otherwise, we violate the maximum principal again/the annulus was really a disk.}
\end{proof}

\begin{remark}
The main result of \cite{HuangWang} essentially states that if a hyperbolic mapping torus $M$ has a short curve (necessarily isotopic to a simple curve in the fiber) whose boundary torus lies is a \emph{specific} compact subset of the moduli space of flat tori then a least area minimal embedding in the homology class of the fiber stays some definite distance from the short curve.  
We obtain the same conclusion if the curve is short enough, but this threshold depends on the shape of the boundary tube.
\end{remark}

Now we can apply Proposition \ref{prop:nodeepness} for geometrically convergent sequences of hyperbolic manifolds where the thin regions converge to rank-$2$ cusps.

\begin{theorem}\label{thm:cusprepelling}
Let $M_n$ be a sequence of complete hyperbolic $3$-manifolds without rank-$1$ cusps converging geometrically to $M_\infty$, where if components $M_n^{<\epsilon_3}$ converge to a cusp, then the rank is $2$, and there are only finitely many such limits in $M_\infty$. For each $n$, assume we have an embedded incompressible closed minimal surface $\Sigma_n\subset M_n$ of bounded genus  that is an area minimizer in its homotopy class. Then there exists $d>0$ such that $\Sigma_n$ is contained in a $2d$-neighborhood of the $\epsilon_3$-thick part of $M_n$ for all $n$ sufficiently large. 
\end{theorem}
\begin{proof}
Consider a rank-$2$ component $C$ of $M_\infty^{<\epsilon_3}$ approximated by compact Margulis tubes $\mathbb T_n(\epsilon_3) \subset M_n$.  Then geometric convergence $M_n\overset{geom}{\longrightarrow}M_\infty$ tells us that the boundary tori $T_n = \partial \mathbb T_n(\epsilon_3)$ converge in the moduli space of Euclidean tori to $T=\partial C$, so that given $\epsilon>0$, the area $A_n$ of $\partial T_n$ is within $\epsilon$ of the area $A$ of $\partial T$, for large enough $n$.

Let $\lambda_n$ and $\omega_n$ be the complex length and Teichm\"uller parameter of the core geodesic and boundary torus of $\mathbb T_n(\epsilon_3)$, respectively.  We also have the scale parameter $t_n$ which determines the area $A_n$ of $T_n$ by $A_n = t_n^2Im(\omega_n)$, as in the proof of Proposition \ref{prop:nodeepness}.

Since the radii $r_n$ of the tubes $\mathbb T_n(\epsilon_3)$ tend to infinity with $n$, by \cite[Lemma 3.2]{Minsky10} we have 
\begin{align*}|\lambda_n - 2\pi i/\omega_n| &= (1-\tanh(r_n)) Re(2\pi i/\omega_n)\\
&\approx 2e^{-2r_n}\frac{2\pi Im(\omega_n)}{|\omega_n|^2}=2e^{-2r_n}\frac{2\pi Im(\omega_n)t_n^2}{4\pi^2\sinh^2(r_n)}\\
&\approx e^{-4r_n} \frac{(A+\epsilon)}{\pi}.
\end{align*}
In particular, this implies that $\omega_n = 2\pi i/\lambda_n + O(e^{-4r_n}) $.

Any preimage to $\HH^2$ of the conformal parameter of $T$ in the moduli space has imaginary part at most $c$. Thus, given $\epsilon>0$,  we have \[ \frac{2\pi  Re\lambda_n}{|\lambda_n|^2} <c+\epsilon +O(e^{-4r_n}),\] for $n$ large enough.
Rearranging the terms and noting that since $Re(\lambda_n) \to 0$, $Re(\lambda_n)^2$ is negligible compared to $Re(\lambda_n)$, we can recover
\[\frac{Im(\lambda_n)^2}{Re(\lambda_n)} >\frac{2\pi}{c+1} = K,\]
for $n$ large enough.

There are finitely many rank-$2$ parabolic cusps in $M_\infty$, so we may apply Proposition \ref{prop:nodeepness} to obtain $d$ such that $\Sigma\cap M_n^{<\epsilon_3}$ is within distance $2d$ of $\partial M_n^{<\epsilon_3}$.
\end{proof}
\begin{remark}
It is immediate from the proof of Theorem \ref{thm:cusprepelling} that the conclusion still holds if we allow infinitely many rank-$2$ parabolic cusps in $M_\infty$, as long as we assume that the Euclidean structures on the boundaries of the rank $2$ cusps lie in a compact set of the ($3$-dimensional) moduli space of Euclidean tori. 
\end{remark}

\begin{remark}
We may also relax the hypotheses of Theorem \ref{thm:cusprepelling} to include the case that $M_\infty$ has infinity many cusps if we require that the minimal surfaces $\Sigma_n$ stay at bounded distance from the basepoint for geometric convergence $M_n\to M_\infty$.  Indeed, we claim that $\Sigma_n$ can only ever meet a (uniformly) bounded number of components of $M_n^{<\epsilon}$, to which we can apply the proof of Theorem \ref{thm:cusprepelling}.  %, and each component has (uniformly) bounded diameter.

First we argue that each component of $\Sigma_n\cap M_n^{\ge \epsilon}$ contributes a definite amount of area to $\Sigma_n$.  To see this, first note that there is a definite distance between $\epsilon$-Margulis tubes and goes to infinity as $\epsilon$ goes to $0$ (see \cite[\S 3.2.2]{Minsky10}).  Taking $\epsilon$ smaller if necessary, we may assume that the distance between distinct tubes is at least $\epsilon$.  Now, find a maximal $\epsilon$-separated set $F$ in $\Sigma_n\cap M_n^{\ge \epsilon}$.  By Anderson's monotonicity formula \cite{Anderson82},  the intersection of each (embedded) $\epsilon/2$-ball centered at a point of $F$ with $\Sigma_n$ contributes at least $4\pi \sinh^2(\epsilon/4)$, the area of a disk of radius $\epsilon/2$ in the hyperbolic plane, to the area of $\Sigma_n$, hence at least $|F|4\pi \sinh^2(\epsilon/4)$ total area to $\Sigma_n$.  

Since the area of $\Sigma_n$ is at most $2\pi|\chi(\Sigma_n)|$ by Gauss-Bonnet, it follows that $F$ has uniformly bounded cardinality.  Since the $\epsilon$-balls centered at $F$ cover $\Sigma_n\cap M_n^{\ge \epsilon}$, the sum of diameters of each component is bounded by $2|F|\epsilon$. Since the number of components of $\Sigma_n\cap M_n^{\ge \epsilon}$ is uniformly bounded,  $\Sigma_n\cap M_n^{\ge \epsilon}$ has bounded distance from the basepoint, and hence it can only intersect finitely many of the $\epsilon$-tubes in $M_n$.  

Since $M_n\to M_\infty$, the boundary of the $\epsilon$-tubes meeting $\Sigma_n$ lie in a compact set of the moduli space, and we can apply the proof of Theorem \ref{thm:cusprepelling} to obtain the result at the beginning of the remark.  
\end{remark}
Theorem \ref{thm:cusprepelling} says that there exists $\epsilon_3\ge\epsilon>0$ so that for $n$ large we have $\Sigma_n\subseteq M_n^{>\epsilon}$. In fact, this holds for any $\epsilon$ satisfying $\log(\frac{\epsilon_3}{\epsilon})\geq 2d+c_3$ for some constant $c_3$ (see \cite{Minsky10}). Since $\Sigma$ is incompressible, this means that the injectivity radius of the intrinsic metrics of $\Sigma_n$ is uniformly bounded below.

Following \cite{Uhlenbeck83}, \cite{Taubes04}, we say that $(\Sigma,g,\alpha,u)$ is a minimal surface in hyperbolic geometry if

\begin{itemize}
    \item $g$ is a hyperbolic metric on $\Sigma$.
    \item $\alpha$ is a holomorphic quadratic differential on $(\Sigma,g)$
    \item $u:\Sigma\rightarrow\mathbb{R}$ is a smooth positive function.
    \item The 2-form $Re(\alpha)$ satisfies the Gauss equation with respect to $e^{2u}g$ and $\mathbb{H}^3$. Explicitly
    \[\Delta u + 1 - e^{2u} - |\alpha|^2e^{-2u}=0
    \]
\end{itemize}

Morally, $\Sigma$ represents a minimal surface with second fundamental form $Re(\alpha)$ in a hyperbolic 3-manifold. The conditions above say that the universal cover $(\hat\Sigma,\hat g,\hat\alpha,\hat u)$ has a immersion to $\mathbb{H}^3$ as a minimal plane invariant by the isometries given by a representation $\pi_1(\Sigma)\rightarrow {\rm PSL}(2,\mathbb{C})$. We will still denote $(\Sigma,g,\alpha,u)$ by $\Sigma$ for simplicity. We can also talk about the index of the second variation of area, and in particular define if a minimal surface $\Sigma$ is stable or not. Well-know results (see \cite{SchoenSimonYau75}) say that for a stable minimal surfaces, a lower bound on injectivity radius implies an upper bound on the norm of the second fundamental form. Such results apply to this definition of minimal surfaces in hyperbolic geometry, and the following compactness result follows.

\begin{theorem}\label{thm:bounded_geometry_converge}
Let $\Sigma_n$ a sequence of stable minimal surfaces in hyperbolic geometry with a positive lower bound on their injectivity radius and bounded genus. Then, up to subsequence, there is a stable minimal surface $\Sigma_\infty$ such that $\Sigma_n \overset{C^{\infty}}{\longrightarrow} \Sigma_\infty$.
\end{theorem}
%\textcolor{red}{Alternative:}
%\begin{theorem}
%If $(M_n, p_n)$ are (pointed) complete hyperbolic 3-manifolds converging geometrically to $(M,p)$, and $(\Sigma_n, p_n) \subset (M_n,p_n)$ are (pointed) closed stable minimal surfaces with upper genus bounds and lower injectivity radius bounds, then there is a stable minimal surface $(\Sigma_\infty, p) \subset (M,p)$ with $(\Sigma_n, p_n)$ converging geometrically to $(\Sigma_\infty, p)$.
%\end{theorem}

Since we are dealing with minimal surfaces in hyperbolic space, this surface has negative curvature bounded by $-1$. In particular, given Theorem \ref{thm:cusprepelling}, we know that $\pi_1$-injective area minimizers along a geometrically convergent sequence of hyperbolic 3-manifolds have limits in the limit manifold.

\section{Minimal surfaces with bounded curvature}\label{sec:AF}

In her seminal work, Uhlenbeck \cite{Uhlenbeck83} describes the space of stable minimal surfaces in hyperbolic geometry, and more precisely, the space of almost-Fuchsian surfaces, detailed in the following theorem.

\begin{theorem}[\cite{Uhlenbeck83}]\label{thm:AF}
If $M$ is a complete hyperbolic $3$-manifold and $\Sigma\subset M$ is a closed minimal surface with principal curvatures $|k_{1,2}(x)|\leq 1$ then
\begin{enumerate}
    \item ${\rm exp}\, T^\perp\Sigma \simeq \widetilde{M} \rightarrow M$, where $\widetilde{M}$ is the covering of $M$ associated to $\pi_1(\Sigma)$.
    \item\label{item:strict} If $|k_{1,2}| < 1$, then $\widetilde{M}$ is quasi-Fuchsian.
    \item\label{item:maxprinciple} $\Sigma$ is area-minimizing and is the only closed minimal surface in $\widetilde{M}$.
    \item $\Sigma\subset\widetilde{M}$ is embedded.
    \item $\Sigma\subset M$ is totally geodesic if and only if $\widetilde{M}$ is Fuchsian.
\end{enumerate}
\end{theorem}

\begin{remark}
We observe that statement (\ref{item:strict}) was only known for $|k_{1,2}|$ strictly bounded by $1$, but in recent work Huang and Lowe \cite{HuangLowe21} that the same conclusion holds for $|k_{1,2}|\leq1$. All other items (including (\ref{item:maxprinciple})) have essentially the same proof for either $|k_{1,2}(x)|\leq 1$ or $|k_{1,2}(x)|< 1$. Commonly the notation of almost-Fuchsian refers to the case when $|k_{1,2}(x)|< 1$.
\end{remark} 

A consequence of this result is that every point of $p\in\Sigma$ has a ``unique" outer tangent horosphere $H_p$. More precisely, lifting to universal covers $\widetilde{\Sigma}\subset\mathbb{H}^3$, for each choice of normal outer direction and for each $\widetilde{p}\in\widetilde{\Sigma}$, there exist a horosphere $H_{\widetilde{p}}$ outer-tangent to $\widetilde{\Sigma}$ at $\widetilde{p}$. This means that is tangent on the side of the chosen normal, and $\widetilde{\Sigma}$ does not intersect the interior of $H_{\widetilde{p}}$.  Moreover, this is a one-to-one correspondence, meaning that $\widetilde{p}$ is the only point of contact for $H_{\widetilde{p}}$,   i.e. $\widetilde{\Sigma}\cap H_{\widetilde{p}} = \lbrace \widetilde{p} \rbrace$. %and such family of horospheres is invariant by the action of $\pi_1(\Sigma)$. 
We would like to describe how this behavior translates for the boundary of almost-Fuchsian manifolds.

\begin{prop}
Let $M$ be a complete hyperbolic 3-manifold and $\Sigma\subset M$ a complete surface that has principal curvatures bounded by $1$ in size ($|k_{1,2}|\leq 1$). Then
\begin{enumerate}
    \item ${\rm exp}\, T^\perp\Sigma \simeq \widetilde{M} \rightarrow M$, where $\widetilde{M}$ is the covering of $M$ associated to $\pi_1(\Sigma)$.
    \item $\Sigma\subset \widetilde{M}$ is embedded.
\end{enumerate}
\end{prop}
\begin{proof}
As in \cite{Uhlenbeck83} we see that if $g$ denotes the induced metric in $\Sigma$ and $B$ denotes its shape operator, then we have the explicit hyperbolic metric in $\Sigma_x\times \mathbb{R}_t$

\[    G(x,t)(u,v) = g(x)(\cosh{t}.u+\sinh{t}.Bu,\cosh{t}.v+\sinh{t}.Bv) + rs,\quad (u,r),\,(v,s)\in T_{(x,t)}\Sigma\times\mathbb{R}
\]

The condition $|k_{1,2}|\leq 1$ allows us to see that the metric $G$ is positive definite, while the completeness of $\Sigma$ allows us to see that $(\Sigma_x\times \mathbb{R}_t, G)$ is complete. As in \cite{Uhlenbeck83} one can verify that the metric $G$ is hyperbolic and that $t=0$ is isometric to $\Sigma$. This proves both items.
\end{proof}

%Let us see how to conclude in this general version the statement about outer tangent horosphere. 
We now describe the behavior of outer tangent horospheres when $|k_{1,2}|\le 1$.
Take $\widetilde{\Sigma}$ the universal covering of $\Sigma$. Since it is complete, we know that ${\rm exp}\, T^\perp\widetilde\Sigma$ is isometric to $\mathbb{H}^3$. 
Hence the normal geodesics to $\widetilde{\Sigma}$ are mutually disjoint. For a choice of outer normal direction and each $ p\in\widetilde{\Sigma}$, we can follow the normal geodesic at $p$ along the outer normal to find a point at infinity $z(p)$. Then the horosphere $H_p$ centered at $z(p)$ passing through $p$ is tangent to $\widetilde{\Sigma}$. To see that $H_p$ is outer tangent, i.e.,  that $\widetilde{\Sigma}$ is disjoint from the horoball bounded by $H_p$, foliate the horoball by the geodesic spheres tangent at $p$. If $\widetilde{\Sigma}$ has a point in the interior of $H_p$, then one of the geodesic spheres from the foliation is tangent to $\Sigma$ at a point different from $p$. Such point would have a normal geodesic intersecting the one emanating from $p$. 

Note that, when $|k_{1,2}( p)| =1$, we cannot necessarily conclude that $\widetilde{\Sigma}\cap H_{p} = \lbrace {p} \rbrace$; see Proposition \ref{prop:parabolic_curve}.
%As for the equality , if we have the strict inequality $|k_{1,2}|< 1$ we postpone it until a later proposition.

As Uhlenbeck points out, it was known by Bianchi and others that if $\Sigma$ is a minimal surface in hyperbolic geometry, then the second fundamental form of $\Sigma$ is the real part of a holomorphic quadratic differential (with respect to the conformal structure induced in $\Sigma$). Then if $\mu$ denotes a conformal structure on $\Sigma$ with hyperbolic metric $g^\mu$, and $\alpha$ is a quadratic differential with respect to $\mu$, Uhlenbeck  \cite[Theorem 4.2]{Uhlenbeck83} describes the Gauss-Codazzi equations for a minimal surface $\Sigma$ in hyperbolic geometry with induced metric $g=e^{2u}g^\mu$:
\begin{equation}\label{eq:GaussCodazzi}
    \Delta_{g^\mu} u + 1 -e^{2u} - |\alpha|_{g^\mu}^2e^{-2u} = 0
\end{equation}

By solving this equation with $g^\mu$ and $\alpha$ as coefficients, Uhlenbeck describes the space of \textit{stable} minimal surfaces. We say that a minimal surface is stable if the second variation of area is a non-negative operator. Uhlenbeck notices that the linearization of \eqref{eq:GaussCodazzi} coincides with the second variation of area and, by applying the Implicit Function Theorem and maximum principle, shows \cite[Theorem 4.4]{Uhlenbeck83} that the space of stable minimal surfaces is star-shaped with respect to the parameter $\alpha$. This means that for given $(g^\mu,\alpha)$ there exists $t_\alpha$ such that the following hold:

\begin{itemize}
    \item If $0\leq t\leq t_\alpha$ then the equation \eqref{eq:GaussCodazzi} for $(g^\mu,t\alpha)$ has a unique stable solution.
    \item Any stable minimal surface is found in this way.
    \item As $t$ increases, the principal curvatures $\pm k(x,t)$ of non-umbilic points strictly increase in size.
\end{itemize}
From this discussion, we can see that \eqref{item:strict} in Theorem \ref{thm:AF} can be replaced by
\begin{itemize}
    \item[(2')] \label{item:nonstrict} If $|k_{1,2}|<1$ then $\widetilde{M}$ is quasi-Fuchsian. If $|k_{1,2}|=1$ somewhere, then $\widetilde{M}$ can be obtained as the limit of almost-Fuchsian manifolds, i.e. quasi-Fuchsian manifolds with a minimal surface satisfying the strict inequality $|k_{1,2}|<1$.
\end{itemize}

Thanks to (\ref{item:nonstrict}) we can expand the property of outer tangent horosphere to minimal surface with $|k_{1,2}|\leq 1$. Clearly, having an outer tangent horosphere is a property preserved by closure. What we potentially lose is the one-to-one correspondence between points in the surface and outer-tangent horospheres. We can have multiple points with the same outer-tangent horosphere, but this can only happen under special circumstances.

\begin{prop}\label{prop:parabolic_curve}
Let $\widetilde{\Sigma}$ be an embedded surface in $\mathbb{H}^3$ with principal curvatures $|k_{1,2}|\leq 1$. Then if two points $p$ and $q$ have the same outer tangent horosphere $H$, then $H\cap\widetilde{\Sigma}$ contains the geodesic in $\widetilde{\Sigma}$ joining $p$ to $q$. Moreover, this geodesic segment is also a geodesic in $H$ and a line of curvature of $\widetilde{\Sigma}$.
\end{prop}

\begin{proof}
Take a geodesic segment $\gamma\subset \widetilde \Sigma$ joining $p$ to $q$. Then the geodesic curvature of $\gamma$ is also bounded in norm by $1$. Suppose for sake of contradiction that $\gamma$ is not contained in $H$. Take in $H$ a disk $D$ that contains $H\cap\gamma$ in its interior. Consider the family of geodesic balls whose intersection with $H$ is $\partial D$ and which contain $\gamma$ in their interiors. By construction, the first geodesic ball to make contact with $\gamma$ will do so tangentially at an interior point of $\gamma$. But geodesic balls have curvature above one, which contradicts the curvature bound on $\gamma$. From this we conclude that $\gamma$ is contained in $H$.

Since the principal curvatures of $H$ are equal to $1$ (in particular $H$ is umbilic), any curve that is not a geodesic will have curvature above one. Hence then $\gamma$ is also a geodesic in $H$, so it has curvature equal to $1$ at all points. This means that in the direction of $\gamma'$ the surface $\widetilde{\Sigma}$ has curvature $1$. Since this is the maximum possible value for a curvature on $\widetilde{\Sigma}$, then $\gamma$ always points in a direction of maximal curvature.
\end{proof}

In the proof of Theorem \ref{thm:drilling_a_curve}, we see as a consequence of Proposition \ref{prop:parabolic_curve} that if $\Sigma\subset M$ has principal curvatures $|k_{1,2}|\le 1$ and $M$ has a parabolic cusp represented by a closed curve homotopic into $\Sigma$, then the geodesic representative of that curve in $\Sigma$ is a horocyclic segment.

\section{Mapping torus 3-manifolds with small geodesics}\label{sec:mapping torus}
In this section, we bound the number of orthogonal intersections of closed curves in a finite type surface equipped with a metric of negative curvature and finite area.  Together with Proposition \ref{prop:parabolic_curve} and the results of Section \ref{sec:minimizers}, we then use this bound to produce families of area minimizing minimal surfaces in fibered hyperbolic $3$-manifolds with maximum principal curvature bounded uniformly away from $1$.  We do this by drilling out families of curves homotopic to a fiber which would be forced to intersect orthogonally on a limit of area minimizing minimal surfaces with maximum principal curvatures tending to $1$.

A \emph{curve system} $\mathcal C\subset \Sigma$ is a non-empty collection of homotopically distinct essential simple closed curves on $\Sigma$.  We say two curve systems $\mathcal C_1, \mathcal C_2\subset \Sigma$ \emph{fill} a homotopically essential subsurface $F=F(\mathcal C_1, \mathcal C_2)\subset\Sigma$ if every essential closed curve in $F$ has an essential intersection with a component of $\mathcal C_1$ or $\mathcal C_2$. 

Let $S_{g,n}$ be an oriented surface of genus $g$, $n$ boundary components or punctures, and negative Euler Characteristic.  Let $m$ be a Riemannian metric of finite area on $S_{g,n}$ so that each boundary component is totally geodesic and with pinched negative curvature, i.e. there are $a,b>0$ with $-a\le\kappa(m)\le -b$.  
The following bounds are obtained by combinatorial Euler characteristic arguments and several applications of Gauss-Bonnet.
\begin{lem}\label{lem:intersections}
Let $\mathcal C_1$ and $\mathcal C_2$ each be curve systems represented by geodesics in $S_{g,n}$ not corresponding to boundary components and sharing no common curves.  Suppose $\mathcal C_1\cup \mathcal C_2$ is connected and every point of intersection $\mathcal C_1 \cap \mathcal C_2$ is orthogonal.  Then the total intersection satisfies \[|\mathcal C_1 \cap \mathcal C_2|\le(\frac{2a}b+1)|\chi(F(\mathcal C_1, \mathcal C_2)|\le(\frac{2a}b+1)|\chi(S_{g,n})|.\]
\end{lem}

\begin{proof}
We have assumed that $\mathcal C_1\cup \mathcal C_2$ is connected, so $\mathcal C_1\cup\mathcal C_2$ fills some homotopically essential connected subsurface $F\subset S_{g,n}$; that is, $F$ is the smallest subsurface so that $F\setminus \mathcal C_1\cup \mathcal C_2$ is a collection of disks with at most one puncture and boundary parallel annuli.  Our assumptions tell us that $F$ is not an annulus and $\chi(F)<0$, because some component of $\mathcal C_1$ meets a component of $\mathcal C_2$ transversely (so $\pi_1(F)$ is non-elementary); we may realize $F$ in $S_{g,n}$ so that its boundary (if any) is totally geodesic.

We consider the components  $\{F_1, ..., F_k\}= F\setminus (\mathcal C_1\cup \mathcal C_2)$. Each has right angled piecewise geodesic polygonal boundary $P_i$.  The edges of $P_i$ alternate between segments of curves of $\mathcal C_1$ and $\mathcal C_2$, so that the number of sides $|P_i|=2p_i$ is even.  A compact annular component has only one polygonal boundary component; the other is totally geodesic and smooth.  

The decomposition of $F$ as  $F\setminus (\mathcal C_1\cup\mathcal C_2)$ almost gives us a cellulation of $F$; we add in an (open) edge joining each puncture to a vertex of its polygonal boundary as well as a vertex on the boundary components of a compact annular component and an edge connecting it to a vertex on its polygonal boundary.  Now we compute $\chi(F)$ in terms of our cellulation.   A disk or punctured disk component $F_i$ contributes $2p_i$ vertices and each compact annular compononent $F_i$ contributes $2p_i+1$ vertices, but each vertex corresponding to a point of intersection of $\mathcal C_1\cap \mathcal C_2$ is counted $4$ times.  This is because the cells all meet at right angles.  If $A$ is the number of compact annular components, then the number $v$ of vertices is given by 
\[v = \sum_{i=1}^k\frac{2p_i}4 +A =|\mathcal C_1\cap \mathcal C_2|+A. \]  
Similarly, if $C$ is the number of cusps components, we compute the number $e$ of edges
\[e = 2|\mathcal C_1\cap \mathcal C_2|+2A+C.\]
The number of faces is $k= A + C +D$, where $D$ is the number of disks, so that
\[|\mathcal C_1\cap \mathcal C_2| -D= -v+e-f = -\chi(F),\]
and in particular, \[|\mathcal C_1\cap \mathcal C_2|= |\chi(F)| +D.\]  

We just need to bound the number $D$ of disks. Reordering if necessary (when there are annular components of $F\setminus (\mathcal C_1\cup \mathcal C_2)$), they are $\{F_1, ..., F_D\}$.
We note that $\sum_{i =1}^D\Area(F_i)\le \Area(F)$.  Using the upper bound $-b$ on curvature and the fact that the boundary of $F$ is totally geodesic, Gauss--Bonnet yields
\[ \Area (F) \le \frac{2\pi|\chi(F)|}b. \]

Using the lower curvature bound, for any $F_i$, Gauss--Bonnet gives
\[\frac\pi2|P_i|-2\pi\chi(F_i)\le a\Area(F_i). \]
The universal cover of $S_{g,n}$ is contractible and negatively curved, so there are no right angled $4$-gons.  Thus, if $F_i$ is a disk, then $p_i\ge 3$ and $\chi(F_i) = 1$.  If $F_i$ is annular, then $\chi(F_i) = 0$ and $p_i\ge 2$.  Thus
\[\Area (F_i)\ge \frac\pi a (p_i-2\chi(F_i))\ge\pi/a.\]

Now summing over disks $\{F_1, ..., F_D\}$  we obtain $\Area(F)\ge D\pi/a$.  Combining this with our upper bound on $\Area(F)$, we get
\[D\le \frac{2a|\chi(F)|}b.\]

Note that that $|\chi(F)|\le |\chi(S_{g,n})| =|2g-2+n|$, so that combining the above estimate with $|\mathcal C_1\cap \mathcal C_2|= D+|\chi(F)|$ completes the proof of the lemma.
\end{proof}

Note that this bound is sharp: if $S$ is a closed surface of genus $g$ and has a metric of constant curvature equal to $1$, then one can build a surface with a pants decomposition and `dual' pants decomposition where all curves meet at right angles, and the union is connected.  The number of intersections of these two pants decompositions is $6(g-1) = 3|\chi(S)|$.

We now construct infinite families of closed hyperbolic $3$-manifolds fibered over the circle such that any area minimizing minimal surface homotopic to a fiber (of a particular fibration) have maximal principal curvatures bounded strictly away from $1$ in absolute value.  
Our approach is consider a family of surface bundles obtained by Dehn filling on a finite volume hyperbolic manifold with cusps.  For filling with large enough slope, we know that area minimizing minimal surfaces must stay away from the tubes that degenerate to cusps.  Thus the thin parts act as barriers to area minimizing surfaces of bounded genus.

Let $\psi:\Sigma\to \Sigma$ be a pseudo-Anosov mapping class.  For a curve system $\mathcal C\subset \Sigma$, we say that $(\psi, \mathcal C)$ \emph{intersects enough} if
\begin{itemize}
    \item $\mathcal C$ and $\psi(\mathcal C)$ share no common curves; and
    \item $F = F(\mathcal C, \psi(\mathcal C))\subset \Sigma$ is connected; and
    \item $i(\mathcal C, \psi(\mathcal C))\ge 5|\chi(F)|$.
\end{itemize}
Note that for some power of $\psi$, any non-empty curve system satisfies the above properties with $F = \Sigma$.  

We take a simple closed curve $\mathcal C =\{\gamma\}$ such that $(\psi, \mathcal C)$ intersects enough.  Further, we require that the geodesic representative $\mathcal C^*$ of $\mathcal C$ in the hyperbolic metric on the mapping cylinder $C_\psi$ is unknotted, i.e. isotopic into the fiber. Note that postcomposing $\psi$ by a large Dehn twist $T_{\mathcal C}^{N}$ in $\mathcal C$, we can assume that $\mathcal C^*$ is very short, hence uknotted by an argument of Otal \cite{Otal03}.

We consider the open manifold $M = C_\psi\setminus \mathcal C^*$.  By Thurston's hyperbolization theorem, $M$ admits a (unique) complete hyperbolic metric of finite volume, and Dehn fillings $M_{\psi, k}$ with slopes $k$ near enough to infinity converge geometrically to $M$.  Moreover, $M_{\psi,k}$ can be realized as the mapping cylinder of $T_{\mathcal C}^k\circ \psi$, hence fibers over $S^1$ with fiber $\Sigma$.

\begin{theorem}\label{thm:drilling_a_curve}
Let $M_{k} = M_{\psi, k}$ be as above with $ k\to  \infty$, and let $\Sigma_k\subset M_k$ be area minimizing minimal surfaces in the homotopy class of the fiber $\Sigma\hookrightarrow M_k$.  Then there exist $K_0$ and $\mu >1$ such that for $k\ge K_0$, the maximum principal curvature of $\Sigma_k$ is bounded below by $\mu$. 
\end{theorem}
\begin{proof}
By Theorem \ref{thm:cusprepelling}, $\Sigma_k$ stays in the complement of  $M_k^{<\epsilon}$ for some $\epsilon\le \epsilon_3$ and $k\ge K_0$ large enough.  Since $\Sigma_k$ is incompressible and $1$-Lipschitz, $\epsilon$ provides a lower bound on the injectivity radius of $\Sigma_k$.  Assume for sake of contradiction that some subsequence (with the same name) has maximum principal curvatures tending to $1$.  

By Theorem \ref{thm:bounded_geometry_converge}, there is a stable minimal surface $\Sigma_{\infty}\subset M$ such that, up to subsequence (again with the same name) $\Sigma_k\xrightarrow{C^\infty}\Sigma_\infty$, and all principal curvatures of $\Sigma_\infty$ are bounded in absolute value by $1$.  Clearly, $\Sigma_\infty$ does not enter the $\epsilon$-thin part of $M$.  %Using $C^\infty$ approximation maps $\phi_k: M^{>\epsilon}\to M_k^{>\epsilon}$ for large enough $k$, we can find minimal surfaces $\Sigma_k' \subset M_k$ (which coincide with $\Sigma_k$ for infinitely many $k$) such that $\Sigma_k'\xrightarrow{C^\infty}\Sigma_\infty$. 

%For sake of contradiction, we assume that the largest principal curvatures of $\Sigma_k'$ tend to $1$. Then $\Sigma_\infty$ has all principal curvatures bounded in absolute value by $1$.
The subgroup $\Gamma\le \pi_1 M$ corresponding to $\Sigma_\infty$ defines a covering space $\widetilde M$ to which $\Sigma_\infty$ lifts isometrically to a homotopy equivalence.  Moreover, $\widetilde M\cong \Sigma\times \RR$ is geometrically finite; by changing the orientation of $\Sigma_\infty$ if necessary, $\mathcal C$ and $\psi(\mathcal C)$ correspond to rank-$1$ parabolic cusps in the `top' and  `bottom' conformal surfaces at infinity, respectively.  %That is, curves corresponding to  $\mathcal C$ and $\psi(\mathcal C)$ on $\Sigma_\infty\subset M_\Gamma$ correspond to conjugacy classes of parabolic elements of $\Gamma$.  \textcolor{red}{Is there an issue with `wrapping' here?  Do we really know that all curves of $\mathcal C$ are on the top and all curves of $\psi(\mathcal C)$ are on the bottom?  Could $\Sigma_\infty$ weave in and out?}

Now we lift the situation to universal covers to obtain a $\Gamma$-equivariant embedded minimal surface $\widetilde \Sigma_\infty\subset \HH^3$.  Suppose $\gamma\in \Gamma$ represents (the conjugacy class of ) $\mathcal C$ and let $z$ be the parabolic fixed point of $\gamma$.  There is a smallest horoball such that the horosphere $H$ comprising its closure meets $\widetilde \Sigma_\infty$ in a set containing a point $p$.  By equivariance, $\gamma. p\in \widetilde \Sigma_\infty\cap H$, so by Proposition \ref{prop:parabolic_curve}, the intersection  $\widetilde{\Sigma}_\infty\cap H$ is \begin{itemize}
    \item the $H$-geodesic line in $H$ joining $p$ to $\gamma. p$; and
    \item a lift of the the geodesic in $\widetilde{\Sigma}_\infty$ in the homotopy class of $\gamma\subset \Sigma_\infty$; and
    \item a line of curvature of $\Sigma_\infty$.
\end{itemize}

Now, we apply the same argument for $\psi_*\gamma$ representing  $\psi(\mathcal C)$ to see that any lift of the geodesic representative of $\psi_*\gamma$ in $\Sigma_\infty$ is a line of curvature (with curvature of opposite sign).  It follows that all points of intersection of the $\Sigma_\infty$-geodesic representatives $\mathcal C$ and $\psi(\mathcal C)$ are orthogonal.

However, $(\psi, \mathcal C)$ intersects enough, which contradicts Lemma \ref{lem:intersections}.  This contradiction provides a lower bound $\mu>1$ for the maximum curvature of the local area minimizers $\Sigma_k \subset M_k$ along any subsequence.
\end{proof}

\begin{remark}
Theorem \ref{thm:drilling_a_curve} clearly holds under less restrictive hypotheses.  Namely, one can drill any curve system $(\psi,\mathcal C)$  that intersects enough (or multiple curve systems) as long as one can assure that the rank-$1$ parabolic curves on the two conformal boundaries of the cover associated to any geometric limit $\Sigma_\infty$ have enough (orthogonal) intersections, as in Lemma \ref{lem:intersections}.  Note that as soon as $\mathcal C$ has at least $2$ components, then there are different homotopy classes of embeddings into $M$ which map to fibers of the approximating manifolds $M_k$, which `weave through' the rank-$2$ cusps changing which curves are on `top' and `bottom' in the corresponding covering space.
\end{remark}

In our next application, we show how a short curve can act as a barrier to find more than one local area minimizing minimal surface in the same homotopy class with maximum prinicpal curvature bounded away from $1$.
\begin{theorem}
Let $\psi: \Sigma \rightarrow \Sigma$ a pseudo-Anosov map, $C_\psi=\Sigma\times[0,1]/{(x,0)\sim(\psi(x),1)}$ its mapping torus, and $\alpha\subseteq \Sigma\times\lbrace0\rbrace,\beta\subset\Sigma\times\lbrace\frac12\rbrace$ two simple closed curves that have non-trivial intersection when homotoped to a single copy of $\Sigma$. Denote by $M_0=C_\psi\setminus\lbrace\alpha,\beta\rbrace$. Then for $n$ sufficiently large, the slope-$n$ Dehn-surgery around $\beta$, denoted by $M_n$, has the property that for any sequence of pseudo-Anosov maps $\psi_k:\Sigma\rightarrow\Sigma$ so that $C_{\psi_k}\overset{geom}{\longrightarrow} M_n$, there exists a uniform constant $\mu(n)>1$ so that $\mu$ is a lower bound for the maximal principal curvature on any area minimizer of $C_{\psi_k}$ in the homotopy class of $\Sigma$.
\end{theorem}
\begin{proof}
Note than in $M_0$ we have two distinct copies of $\Sigma$, namely $\Sigma\times\lbrace\frac14\rbrace, \Sigma\times\lbrace\frac34\rbrace$ (differentiated by their cyclic order with $\alpha,\beta$), so we can obtain two area minimizers, one in each homotopy class \cite{FreedmanHassScott83}. 
By Theorem \ref{thm:minimizersurvives} of the Appendix, as $M_n \overset{geom}{\longrightarrow} M_0$, we have that for $n$ large enough we will find local area minimizers homotopic to the standard $\Sigma$. 
Since now they belong to the same homotopy class,  each of them has a point with a principal curvature strictly bigger than $1$. By Theorem \ref{thm:cusprepelling} one of these surfaces is the limit of the area minimizers in $M_n$, which implies the existence of an uniform lower bound $\mu>1$ for the largest principal curvature.
\end{proof}

\begin{remark}
In the previous result, we can also conclude that there is a uniform $\mu>1$ lower bound for the largest principal curvature along any sequence of minimal surfaces that converge to either local minimizer in $M_n$.
\end{remark}
\begin{remark}
Using Proposition \ref{prop:nodeepness}, we can guarantee the existence of two local area minimizers using the short curve $\beta$ as a barrier using shrinkwrapping \cite{CalegariGabai06}, as detailed in \cite{Coskunuzer21}.
\end{remark}

\section{Further discussion}\label{appendix}
We make some concluding remarks about existence of local area minimizing surfaces in some families of hyperbolic $3$-manifolds.  
%In particular, we provide examples of tame  hyperbolic $3$-manifolds with infinitely many homotopic incompressible local area minimizers that exit an end of the manifold.  These examples do not cover a finite volume manifold.
While in previous sections we have used information about a sequence $M_n \overset{geom}{\longrightarrow} M_\infty$ to conclude the existence of a minimal surface in the limit, here we include a result that uses a minimal surface in the limit to conclude existence along the sequence. While the arguments are known, we include this discussion for the sake of completeness.

\begin{theorem}\label{thm:minimizersurvives}
Let $M_n \overset{geom}{\longrightarrow} M_\infty$ be a geometrically convergent sequence of hyperbolic 3-manifolds. If $M_\infty$ has a closed embedded local minimizer of area $\Sigma$ that does not extend to a local foliation by minimal surfaces, then for $n$ sufficiently large, there exist local area-minimizers $\Sigma_n\subset M_n$ so that $\Sigma_n \rightarrow \Sigma$.
\end{theorem}
\begin{proof}
We know by \cite[Proposition 3.2]{BrayBrendleNeves10} and \cite[Lemma 10]{Song18} that in a neighbourhood of $\Sigma$, we have a foliation such that, on each side of $\Sigma$, is either strictly mean convex, strictly mean concave, or is foliated by minimal surfaces. By a result of Anderson (see \cite[Theorem 5.5]{Anderson83}) we can further assume that $\Sigma$ is the only minimal leaf. Because $\Sigma$ is a local minimizer of area,  neither side can be strictly mean convex. Hence we can consider $\Sigma^\pm$, homotopic surfaces on each side of $\Sigma$, whose mean curvature vectors point strictly towards $\Sigma$. By geometric convergence, $M_N$ contains regions exiting the end which are $C^2$ close to the region $U$ limited onto by $\Sigma^\pm$, for large $n$. Hence eventually these regions are mean convex, so we can find area minimizers by minimizing area along surfaces contained in those regions. These are the local area minimizers $\Sigma_n\subset M_n$, which we can make converge to $\Sigma$ by taking $\Sigma^\pm$ closer to $\Sigma$ and a standard diagonal argument.
\end{proof}

%\begin{theorem}
%Let $M$ be a bounded geometry hyperbolic end. If $M$ is asymptotically periodic and the periodic geometric limit does not foliate by compact minimal surfaces, then $M$ contains infinitely many local area minimizer minimal surfaces exiting the end.
%\end{theorem}
%\begin{proof}
%Denote by $N=S\times\mathbb{R}$ the periodic geometric limit. Taking $\Sigma$ an area minimizer in the homotopy class of $S$, we know by \cite[Proposition 3.2]{BrayBrendleNeves10}, \cite[Lemma 10]{Song18} that in a neighbourhood of $\Sigma$ we have a foliation such that each side is either mean convex, mean concave or foliates by minimal surfaces. Since $\Sigma$ is an area minimizer, we know that it can't be mean concave, while by assumption neither of these neighbourhood can't foliate (see \cite[Theorem 5.5]{Anderson83}). Hence we can consider $\Sigma^\pm$ homotopic surfaces at each side of $\Sigma$ so that their mean curvature vector points strictly towards $\Sigma$. By geometric convergence, $M$ contains regions exiting the end which are $C^2$ close to the region $U$ limited by $\Sigma^\pm$. Hence eventually these regions are mean convex, so we can find an area minimizers by minimizing area along surfaces contained in those regions. These are the local area minimizers in $M$ that exit the end.
%\end{proof}
We say that an end $E$ of a complete hyperbolic $3$-manifold $M$ is \emph{asymptotically periodic} if for any sequence of points $p_n$ exiting $E$, the geometric limit of $M$ based at $p_n$ converges to the cyclic cover of a hyperbolic mapping cylinder corresponding to the fiber.

We can apply Theorem \ref{thm:minimizersurvives} to conclude
\begin{cor}
Any complete hyperbolic $3$-manifold $M$ with an asymptotically periodic end $E$ has an infinite sequence of pairwise homotopic, $\pi_1$-injective local area minimizing surfaces exiting $E$, as long as the periodic manifold to which the end $E$ converges is not foliated by minimal surfaces exiting the ends.
\end{cor}
%We can apply Theorem \ref{thm:minimizersurvives} to conclude that if we have an asymptotically periodic geometrically infinite hyperbolic end that does not converge to a space that is foliated by minimal surfaces, then such an end has infinitely many local area minimizers exiting it.
\begin{remark}
There exist non-periodic complete hyperbolic $3$-manifolds with an asymptotically periodic end.  Such examples can be obtained, using Thurston's Double Limit Theorem, as limits of quasi-Fuchsian manifolds with conformal end invariants $(X_n, Y_n)$ with $X_n$ staying in a compact subset of the Teichm\"uller space and $Y_n$ converging in Thurston's compactifictaion to the projective class of a measured lamination left invariant by some pseudo-Anosov mapping $\psi$.  
Such a limit has one geometrically finite end, and one geometrically infinite end with a ray of local area minimizing homotopy equivalences that exit the end.
\end{remark}
\bibliographystyle{amsalpha}
\bibliography{mybib}
\Addresses
\end{document}